\documentclass[oneside]{amsart}

\usepackage[letterpaper,body={15.0cm,22.5cm}, mag=1000]{geometry}
\usepackage[utf8]{inputenc}
\usepackage{amssymb}
\usepackage{amsthm}
\usepackage{amscd}
\usepackage{enumitem}
\usepackage{float}
\usepackage{placeins}
\usepackage{caption}
\usepackage{mathtools}

\usepackage{times}
\usepackage{graphicx}
\usepackage{tikz}
\usepackage{tikz-cd}
\usepackage[all,cmtip]{xy}

\numberwithin{equation}{section}
\theoremstyle{plain}

\newtheorem{cor}[equation]{Corollary}
\newtheorem{lemma}[equation]{Lemma}

\newtheorem{pro}[equation]{Proposition}

\newtheorem{exa}[equation]{Example}
\newtheorem{lem}[equation]{Lemma}

\newtheorem{thm}[equation]{Theorem}

\theoremstyle{definition}

\newtheorem{remark}[equation]{Remark}

\newcommand{\dlabel}[1]{\ifmmode \text{\ttfamily \upshape [#1] } \else
{\ttfamily \upshape [#1] }\fi \label{#1}}

\newcommand{\C}{\operatorname{C} }

\newcommand{\Z}{\operatorname{Z} }

\newcommand{\Cb}{\operatorname{Cb} }
\newcommand{\Pb}{\operatorname{Pb} }

\newcommand{\Id}{\operatorname{Id}}
\renewcommand{\Pr}{\operatorname{Pr} }

\newcommand{\gen}[1]{\left < #1 \right >}
\newcommand{\Aut}{\operatorname{Aut} }

\newcommand{\Ker}{\operatorname{Ker} }

\newcommand{\Soc}{\operatorname{Soc} }
\newcommand{\Ann}{\operatorname{Ann} }

\newcommand{\Fix}{\operatorname{Fix} }

\newcommand{\FCI}{\operatorname{FCI}}
\author{Susanta Mondal}
\address{Harish-Chandra Research Institute, A CI of Homi Bhabha National Institute, Chhatnag Road, Jhunsi, Prayagraj-211 019, India}
\email{susantamondal@hri.res.in}

\author{Manoj K. Yadav}
\address{Harish-Chandra Research Institute, A CI of Homi Bhabha National Institute, Chhatnag Road, Jhunsi, Prayagraj-211 019, India}
\email{myadav@hri.res.in}

\subjclass[2010]{16N99, 16T25, 20F24}
\keywords{Skew left brace,  ideal, commutativiy degree, isoclinism, nilpotency}

\begin{document}
\setcounter{page}{1}
\title[commuting probability of skew  left braces]
{commuting probability of  skew  left  braces}

\maketitle

\begin{abstract}
We introduce a concept of the commuting probability of a skew left brace analogous to group theory.
We establish upper and lower bounds for the commuting probability and prove that, for finite non-trivial skew left braces, it is always  at most $\frac{3}{4}$. Interestingly, there is no  skew left brace with commuting probability in the open  interval $(5/8, 1)$, except  $\frac{3}{4}$, for which we construct an explicit example. A characterization of skew left braces having commuting probability $\frac{3}{4}$ or $\frac{5}{8}$ is presented. We further show that the finite skew left braces  with commuting probability  larger than $\frac{65}{128}$ are necessarily nilpotent. We prove that the commuting probability remains invariant under isoclinism of skew braces.  We introduce a concept of a compact Hausdorff topological skew left brace $B$, where we prove that the set of all elements of $B$ having finite centraliser index in $B$ is a Borel subgroup. For such  infinite non-trivial  skew left braces too $\frac{3}{4}$ is the upper bound for the commuting probability, and $\frac{3}{4}$ is the only rational number which occurs as commuting probability in the open interval $(5/8, 1)$.
\end{abstract}

\section{Introduction}

Commuting probability of groups has long been studied in the literature, starting from its genesis in \cite{ET68} by Erdos and Turan, and its inception in \cite{WHG73} by Gustafson  for groups.  This article aims at introducing an analogous concept for skew left braces and obtaining several fundamental results. A set $B$ equipped with two group structures $(B, +)$ and $(B, \circ)$ is said to be a skew left brace if $a \circ (b + c) = (a \circ b) - a + (a \circ c)$  for all $a, b, c \in B$, where $-a$ is the inverse of $a$ in $(B, +)$. It turns out that the group identities of both the groups $(B, +)$ and $(B, \circ)$ coincide, which we denote by $1$.  If $(B, +)$ is an abelian group, the skew left braces is called a left brace.
The concept of left braces and skew left braces were, respectively, introduced by Rump \cite{Rump07} and Guarnieri and Vendramin \cite{GV17} in connection with the study of set-theoretical solutions of the quantum Yang-Baxter equation, suggested by Drienfeld  \cite{Drinfeld92}.

For a finite skew left brace $(B, +, \circ)$, we define commuting probability, denoted by $\Pb(B)$, by the quantity
$$\Pb(B) := \frac{1}{|B|^2} |\{(x, y)  \in B \times B \mid x + y -x -y = x \circ y \circ x^{-1} \circ y^{-1} = -x + (x \circ y) - y = 1\}|,$$
where $x^{-1}$ denotes the inverse of $x$ in $(B, \circ)$. For infinite skew left brace, we first introduce the concept of compact Hausdorff topological skew left brace, and then the concept of a measure on the closed subsets of the skew left brace. To be more precise, we say that a skew left brace $B$  topological  if there exists a common topology of subsets of $B$ under which both the groups $(B, +)$ and $(B, \circ)$ are topological groups. Let $\mu$ be a probability measure on the group $(B, \circ)$ and 
$$Z:= \{ (x, y)  \in B \times B \mid x + y -x -y = x \circ y \circ x^{-1} \circ y^{-1} = -x + (x \circ y) - y = 1\}.$$
Then the  commuting probability of a compact Hausdorff topological skew left brace $B$, again denoted by $\Pb(B)$, is defined by
$$\Pb(B) := (\mu \times \mu) (Z).$$
See Section 5 for more details.

The article is constituted of four sections after the present one. The second section is on preliminaries and some key observations. It is observed that $\Cb_B(x)$, the brace centraliser of $x \in B$ is a subgroup of $(B, \circ)$.  Also the set $\FCI(B)$ of all elements of $B$ whose centralisers have finite index in $(B, \circ)$ is a subgroup of $(B, \circ)$.  Commuting probability for finite skew braces is introduced in Section 3, where several bounds on it are obtained. We prove that if $\Pb(B)$ is different from  $\frac{3}{4}$ and $1$, then $\Pb(B) \in (0, \frac{5}{8}]$. Further,  $\Pb(B) = \frac{3}{4}$ if and only if  $B/\Ann(B)$ is of order $2$, where $\Ann(B)$ denotes the annihilator of the skew brace $B$. It is proved that if  $\Pb(B) = \frac{5}{8}$, then  both the additive and the multiplicative groups of quotient skew brace $B/\Ann(B)$ are isomorphic to $\mathbb{Z}/2\mathbb{Z} \times \mathbb{Z}/2\mathbb{Z}$. It is also proved that $\Pb(B) \leqslant \Pb(A)$ for any sub-skew brace $A$ of a finite skew brace $B$. Connections between commuting probability and  nilpotency of skew left braces is investigated in Section 4. It is proved that a finite skew brace $B$  with $\Pb(B) \ge \frac{65}{128}$ is nilpotent. In the same section, we also show that commuting probability is invariant under isoclinism of skew left braces. In the concluding section, Section 5, we introduce the concept of commuting probability of compact Hausdorff  topological skew left braces, where, among several interesting results,  we prove that $\FCI(B)$ is a Borel set, and $\FCI(B)$ is open in $B$ if and only if $\Pb(B) > 0$. It is also established that  if $\Pb(B)$ is different from  $\frac{3}{4}$ and $1$, then $\Pb(B) \in [0, \frac{5}{8}]$.

We conclude this section by setting some notations. As mentioned above, the common additive and multiplicative identity of a skew left  brace is denoted by $1$. For a skew left brace $(B, +, \circ)$ and $a, b \in B$, the commutator $[a, b]^+ = a+b-a-b$ in $(B, +)$ is denoted by $\gamma^+(a, b)$ and the commutator $[a, b]^{\circ} = a\circ b \circ a^{-1} \circ b^{-1}$ in $(B, \circ)$ is denoted by $\gamma^{\circ}(a, b)$. Left normed higher commutators in $(B, +)$ and $(B, \circ)$ are defined inductively. The center of a group $G$ is denoted by $\Z(G)$, and for any element $g \in G$, the centraliser of $g$ in $G$ is denoted by $\C_G(g)$.

 
\section{Preliminaries and some key results}

A non-empty set $B$, equipped with two group structures  $(B,+)$ and $(B, \circ)$, is  said to be a {\it skew left brace} if for $a, b, c \in B$ the following compatibility condition (skew left distributivity) holds:
 $$a \circ (b + c) = a \circ b - a + a \circ c,$$
 where $-a$ is the inverse of $a$ in $(B, +)$. A {\it skew right brace} can be defined analogously, that is, demanding  skew right distributivity.
 A skew left brace $(B, +, \circ)$ is said to be a {\it two-sided skew brace} if $(B, +, \circ)$ is also  a skew right brace.   Note that if $(B, +, \circ)$ is a skew left brace with $(B, \circ)$ abelian, then $(B, +, \circ)$ is a two-sided skew brace.  
 
The groups $(B,+)$ and $(B,\circ)$ are, respectively, called  {\it additive group} and {\it multiplicative group} of the skew left brace $(B, +, \circ)$.  For simplicity of notation we denote a skew left brace $(B, +, \circ)$ by $B$. If we want to emphasise the role of  group operations, then we use the original notation.   We will normally drop the use of word `left' from skew left brace, and just say  skew brace, unless it is necessary to specify the property.

Every group $(B, +)$ can viewed as a skew  brace by defining $\circ = +$. Such a skew brace is called \emph{trivial}. A skew brace which is not a trivial skew brace is called a \emph{non-trivial skew brace}.
 Easiest among non-trivial skew braces are skew braces $(B, +, +_{op})$ with $(B, +)$ any non-abelian group and $(B, +_{op})$  the opposite group of $(B, +)$, that is, $a +_{op} b = b + a$ for all $a, b \in B$.

   Let $X$ be a given property of  groups.   A skew brace $(B, +, \circ)$ is said to be {\it $X$-type} if its additive group $(B,+)$ satisfies  the property $X$. If $X$ is the abelian property of groups, then a  skew left brace $(B, +, \circ)$ is said to be  \emph{abelian-type}.  An abelian-type skew  brace is called a \emph{brace} in the literature. For simplicity of notation, `an abelian-type trivial skew brace' will be termed as `a trivial brace'.
Let $\Aut(B, +)$ denote the automorphism group of the group $(B, +)$. For a skew brace  $(B, +, \circ)$ and its element  $a$, define a map $\lambda_a : (B, +) \to (B, +)$  by setting 
 $$\lambda_a(b) = - a + (a \circ b)$$ 
 for all $b \in B$. It is now well known that $\lambda_a$ is an automorphism of $(B, +)$ and the  map $\lambda : (B, \circ) \to \Aut(B, +)$; $a \mapsto \lambda_a$ is a group homomorphism.

  Let $(B,  +, \circ)$ be a skew  brace. A subset $H$ of $B$ is said to be a {\it sub-skew  brace} of the skew   brace $(B, +, \circ)$  if $(H, + , \circ)$ is also a skew  brace.   For a given subset $S$ of $B$, the sub-skew brace generated by $S$ is defined to be the smallest sub-skew brace of $B$ containing $S$.  A sub-skew brace $(I, +, \circ)$ of  $(B, +, \circ)$  is said to be a left ideal of $(B, +, \circ)$ if $\lambda_a(I) \subseteq I$ for all  $a \in B$.  A left ideal $(I, +, \circ)$ of $(B, +, \circ)$ is said to be an ideal of $(B, +, \circ)$ if $(I, +)$ and $(I, \circ)$ are normal subgroups of $(B, +)$ and $(B,, \circ)$ respectively.   An ideal generated by a given subset $S$ of $B$ is defined to be the smallest ideal of $B$ containing $S$. It is easy to see that the  kernel of the homomorphism $\lambda$ is given by 
 $$\Ker(\lambda) = \{a \in B \mid a + b = a \circ b \mbox{ for all } b \in B \}.$$

For two skew braces $A$ and $B$, a map $f : A \to B$ is said to be a {\it brace homomorphism} if $f$ is a group homomorphism from the additive group of $A$ to the additive group of $B$ and also from the multiplicative group $A$ to the multiplicative group of $B$. For any brace homomorphism $f : A \to B$, it turns out that $\Ker(f)$ is an ideal of $A$. A bijective brace homomorphism is called a {\it brace isomorphism}.

Let $(B, +, \circ)$ be a skew brace and $a, b \in B$, define 
$$a \ast b := \lambda_a(b) - b = -a + (a \circ b) - b.$$
 The  {\it socle} of $B$,  denoted by $\Soc(B)$, is defined by 
 $$\Soc(B) := \Ker(\lambda) \cap \Z(B,+).$$
The  {\it annihilator} of $B$, denoted by $\Ann(B)$, is defined by 
 $$\Ann(B) := \Ker(\lambda) \cap Z(B,+) \cap \Z(B, \circ) = \Soc(B) \cap  \Z(B, \circ).$$
 It is not very difficult to see that  $\Ann(B)$  is an ideal of $B$. It  follows from \cite[Lemma 2.5]{GV17} or \cite[Proposition 2.9]{BEP24} that   $\Soc(B)$ is also ideal of $B$. We now define higher annihilators of $B$. For $n \ge 2$, define $\Ann_n(B)$, the $n$th annihilator  of $B$, by
$$\Ann_n(B) / \Ann_{n-1}(B) = \Ann\big(B /  \Ann_{n-1}(B)\big).$$
This is possible because $\Ann_{n-1}(B)$ is an ideal of $B$ \cite[Definition 2.2]{JAV23}.  It turns out that 
$$\Ann_n(B) = \{a \in B \mid a*b, b*a,  [a, b]^{+}  \in \Ann_{n-1}(B) \mbox{ for all } b \in B\}.$$
This is an ascending series of ideals of $B$, which is called the  \textit{annihilator series}. We shall view this series  analogous to the upper central series of a group.

Let $B := (B, +, \circ)$ be a skew  brace and $\gamma_n(B, +)$ and  $\gamma_n(B, \circ)$, respectively, denote the $n$th terms of the lower central series of $(B, +)$ and $(B, \circ)$.  For the left ideals $I$ and $J$ of a skew brace $B$, define 
 $$I*J := \langle x*y \mid x \in I, y \in J\rangle^{+},$$ 
 the subgroup generated by the set $\{x*y \mid x \in I, y \in J\}$ in $(B, +)$.
 Set $B^1 = B$ and, recursively, define 
$$B^n :=  B*B^{n-1} = \gen{a*b \mid a \in B, b\in B^{n-1}}^{+}.$$
 It is well known that $B^2$ is an ideal of $B$. But $B^n$, $n \ge 3$, in general is not an ideal of $B$. However, $B^n$ is a left ideal of $B$.  Further, set $B^{(1)} = B$ and, recursively, define 
$$B^{(n)} :=  B^{(n-1)}*B = \gen{b*a \mid a \in B, b\in B^{(n-1)}}^{+}.$$
 It is proved in \cite[Proposition 2.1]{CSV19} that $B^{(n)}$ is an ideal of $B$.
 
 We consider the lower central series of a skew brace defined in \cite{BJ23}. 
  Set $\Gamma_1(B)  := B$ and define
$$\Gamma_n(B) := \gen{B* \Gamma_{n-1}(B), \Gamma_{n-1}(B) * B, [B, \Gamma_{n-1}(B)]^{+}}^{+},$$
where $[B, \Gamma_{n-1}(B)]^{+}$ denotes the subgroup of $(B, +)$ generated by the set $\{[a, u]^{+} \mid a \in B, u \in \Gamma_{n-1}(B)\}$.
The authors of  \cite{BJ23} start with $0$ index, but we start with $1$. $\Gamma_n(B)$ is an ideal of $B$ and  $\Gamma_{n+1}(B) \le \Gamma_{n}(B)$ for all integers  $n \ge 1$. We shall view this series of ideals of a skew left  brace analogous to the lower central series of a group.

A skew brace $B$ is said to be  {\em nilpotent} if there exists an integer $n$ such that  $\Ann_n(B) = B$. The least such integer is called the  {\em nilpotency class} of $B$.  The authors of \cite{BJ23} call such skew braces centrally nilpotent and prove that   $\Ann_n(B) = B$ if and only if $\Gamma_{n+1}(B) = 1$ (\cite[Theorem 2.8]{BJ23}).

We define the  \emph{brace centralizer} of an element $x \in B$ to be the subset 
\begin{eqnarray*}
\Cb_B(x) &:=&  \{b \in A \mid x * b = b * x = [x, b]^{+} = 1\}\\
&= &\{b \in B \mid x * b = [x, b]^{\circ} = [x, b]^{+} = 1\}.
\end{eqnarray*}
The {\it brace centralizer}  of a subset  $S$ of a skew brace $B := (B, +, \circ)$, denoted as $\Cb_{B}(S)$, is defined to be
$$\Cb_{B}(S) := \cap_{x \in S} \Cb_B(x).$$
 Note that if $x \in \Ann(B)$, then $\Cb_{B}(x) = B$. Also note that $\Cb_{B}(B) = \Ann(B)$. It follows from \cite[Proposition 2.19]{SV18} that if $x \in \Cb_{B}(x)$ for all $x \in B$, then $B$ is a two-sided skew brace. The converse need not be true in general (see Example \ref{3/4} below).

Following \cite{CFT25}, for a skew brace $B:=(B, +, \circ)$ and an element $x \in B$,  we define 
$$\Fix^l_B(x)  = \{b \in B \mid \lambda_b(x) = x\} \mbox{ and } \Fix^r_B(x)  = \{b \in B \mid \lambda_x(b) = b\}.$$
We now define the left and the right centralizers, respectively, of an element $x \in B$ as follows:
$$\Cb^l_B(x) := \Fix^l_B(x) \cap \C_{(B, \circ)}(x) \mbox{ and }  \Cb^r_B(x) := \Fix^r_B(x) \cap \C_{(B, +)}(x).$$
It is easy to see that 
$$\Cb_B(x) = \Cb^l_B(x) \cap \Cb^r_B(x).$$
 A straightforward computation shows that $\Fix^r_B(x) \cap \C_{(B, \circ)}(x)$ is a sub-skew brace of $B$.  Note that both $\C_{(B, \circ)}(x)$ and $\Fix^l_B(x)$ are subgroups of $(B, \circ)$. Thus we get the following key result \cite[Theorem 3.3]{CFT25}.

\begin{pro}\label{pro:subgroup}
   Let $(B, +, \circ)$ be a skew brace.    If $x \notin \Ann(B)$, then $\Cb_{B}(x)$ is a proper subgroup of  $(B, \circ)$.
\end{pro}

As remarked in \cite[page 175]{CFT25},  $\Cb_{B}(x)$ is a sub-skew brace of a two sided skew brace $B$. But in general  $\Cb_{B}(x)$ may not be a subgroup $(B, +)$ for each $x \in B$. A GAP search tells that skew braces [8, 18], [12, 15] and several for order 16, including [16, 62-63], [16, 73-76] admit elements $x$ such that $\Cb_{B}(x)$ is not a subgroup $(B, +)$. The case [8,18] is worked out in \cite[Example 3.6]{CFT25}. This skew brace  $B$ is such that its additive group is $\mathbb{Z}/2\mathbb{Z} \times \mathbb{Z}/4\mathbb{Z}$ and its multiplicative group is $D_8$, the dihedral group of order 8. It follows that $\Cb_B(0,3) = \{(0, 0), (1, 1)\}$, which is a subgroup of $(B, \circ)$, but not of $(B, +)$ as $(1, 1) + (1, 1) = (0, 2) \not\in \Cb_B(0,3)$. This example also shows that the element $(0, 3)$ does not lie in its brace centralizer. This too reveals that  $- (1, 1) = (-1, -1) = (1, 3) \not\in \Cb_B(0,3)$. At this point skew brace theory differs from group theory.

A skew left brace $(B, +, \circ)$ is said to be \emph{symmetric} if $(B, \circ, +)$ is also a skew left brace. This notion was introduced in \cite{LC19}, and was further studied in \cite{BNY23, AC20}.  Let $(B, +, \circ)$ be a skew brace. Note that the map $(B, \circ) \to \Aut(B, +)$, $b \mapsto \lambda_b$ is also a map from $(B, +)$ to $\Aut(B, +)$.  We say that the skew brace $(B, +, \circ)$ is \emph{$\lambda$-homomorphic} if this map is also a group homomorphism from $(B, +)$ to $\Aut(B, +)$.  This concept was introduced in \cite{CCD20}, and was further explored in \cite{BNY22}. We now observe

\begin{pro}
    Let $(B, +, \circ)$ be either a symmetric skew brace or a $\lambda$-homomorphic skew brace. Then $\Cb_{B}(x)$ is a sub-skew brace of $B$.
\end{pro}
\begin{proof}
    Let $a, b \in \Cb_{B}(x)$ be two arbitrary elements.   Since $\Fix^r_B(x)$ and $\C_{(B, +)}(x)$ are subgroups of $(B, +)$ for all $x \in B$, it suffices  to prove that  $\Fix^l_B(x)$ is a subgroup of $(B, +)$, that is,  $\lambda_{a + b} (x) = x$ and $\lambda_{(-a)}(x)=x$. For a symmetric skew brace, by the given hypothesis and \cite[Proposition 3.6.]{BNY23}, we get
    $$
        \lambda_{a + b} (x) = \lambda_{b \circ a} (x)=\lambda_b(\lambda_ a(x)) = x
    $$
    and
    $$
       x=\lambda_{b -b} (x)=\lambda_{(-b)\circ b} (x)=\lambda_{(-b)}(\lambda_ b (x))=\lambda_{(-b)}(x).
    $$   
  Also, for  a $\lambda$-homomorphic skew brace, by its definition, we get 
  $$
        \lambda_{a + b} (x) = \lambda_{a \circ b} (x)=\lambda_a(\lambda_ b(x)) = x
    $$
    and
    $$
        x=\lambda_{-b  + b} (x)=\lambda_{(-b)\circ b} (x)=\lambda_{-b}(\lambda_ b (x))=\lambda_{-b}(x).
    $$   
The proof is complete.
\end{proof}

We say that $x$ has \emph{finite centraliser index in $B$}  if  $[(B, \circ) : \Cb_B(x)]$, the  index  of  $\Cb_B(x)$ in $(B, \circ)$ is finite.  Let $\FCI(B)$ denote the set of all elements of  $B$ having finite centraliser index in $B$. The skew brace $B$ is said to be \emph{FCI} if $B = \FCI(B)$. Note that if $[B: \Ann(B)]$ is finite, then $B$ is FCI. So if $B$ is not FCI, then $[B:\Ann(B)]$ is not finite. The skew braces with property (S) were introduced in \cite[Definition 3.7]{CFT25}. An element $x \in B$ is said to be an \emph{$s$-element} if both the indices $[(B, +) : \Cb^r_B(x)]$ and $[(B, \circ) : \Cb^l_B(x)]$ are finite.
A skew brace $B$ is said to have \emph{property (S)} if  each $x \in B$ is an $s$-element. We remark that each $s$-element of $B$ lies in $\FCI(B)$.  However, the two concepts, namely skew brace with property (S) and FCI brace, may not be the same in general. We further remark that, if either  $\Cb^r_B(x)$ or $\Cb_B(x)$  is a left ideal of $B$ for all $x \in B$, then both theses concepts coincide. 

We now prove
\begin{pro}
   For a skew brace $(B,+, \circ)$,  $\FCI(B)$ is a subgroup of $(B,\circ)$. Moreover, if $B$ is two sided, symmetric or $\lambda$-homomorphic, then $\FCI(B)$ is a sub-skew brace of $B$.
\end{pro}
\begin{proof}
 For $a, x_1, x_2 \in B$ such that  $a \in \Cb_B(x_1)\cap\Cb_B(x_2)$, notice that $x_1, x_2 \in \Cb_B(a)$. Thus, by Proposition  \ref{pro:subgroup}, $x_1 \circ x_2 \in  \Cb_B(a)$, which gives   $\Cb_B(x_1) \cap \Cb_B(x_2) \subset \Cb_B(x_1\circ x_2)$. Now, assuming $x_1, x_2 \in \FCI(B)$, we see that  $[(B,\circ) : \Cb_B(x_1\circ x_2)] \leqslant [(B,\circ) : \Cb_B(x_1)][(B,\circ):\Cb_B(x_2)]<\infty$, and therefore  $x_1 \circ x_2 \in \FCI(B)$. Similarly, it is easy to see that  $\Cb_B(x) = \Cb_B(x^{-1})$ for all $x \in \FCI(B)$.  Other assertions can be proved by a routine calculation. The proof is complete.
\end{proof}

For a two-sided brace we can prove the following result, which is analogous to \cite[Proposition 3.10]{CFT25}.  
 \begin{pro}
  Let $(B,+, \circ)$ be a two-sided brace. Then $\FCI(B)$ is an ideal of $B$.
\end{pro}
\begin{proof}
    Set $F:= \FCI(B)$ and  $c \in F$.  Since $B$ is a two-sided brace, each inner automorphism of $(B,\circ)$ is a brace automorphism, and therefore,  $x \circ c \circ x^{-1} \in F$ for all $x \in B$. Thus $F$ is normal in both $(B, +)$ and $(B, \circ)$.  As $\lambda_b(c)=b\ast c+c$ for all $b \in B$ and $F$ is a subgroup of $(B,+)$, it suffices to prove $b\ast c \in F$.  As $B$ is a two sided brace, `$\ast$' is associative on $B$. Thus it follows that the sets $\{x \ast( b\ast c) \mid x \in B\}$ and $\{(b \ast c) \ast x \mid x \in B\}$ are finite. Note that $x\circ(b\ast c)\circ x^{-1} = (x\circ b\circ x^{-1})\ast ( x\circ c\circ x^{-1})$ for all $x \in B$. Since $x\circ c\circ x^{-1}\in F$, $\{x\circ(b\ast c)\circ x^{-1} \mid x\in B\}$ is finite. This completes the proof.
 \end{proof}


\section{commuting probability of finite skew braces}

We start this section by defining  commuting probability of a finite skew  brace. For  a finite skew  brace  $B:= (B, +, \circ)$, we define the  \emph{commuting probability} of $B$, denoted by $\Pb(B)$, to be 
$$\Pb(B)  =  \frac{1}{ |B|^{2}} \big(|\{(a, b) \mid a * b = b* a = \gamma^+(a, b) = 1\}|\big).$$
It is clear that $\Pb(B) = 1$ for all  trivial braces $B$. For any other brace $B$,  $\Pb(B)$ lies in the open interval $(0, 1)$.
 Note that, analogous to group theory, we can derive $\Pb(B)$ in terms of brace centralizers as follows:
$$\Pb(B) = \frac{\sum_{x\in B}|\Cb_{B}(x)|}{|B|^2},$$
which can be further expanded to
\begin{eqnarray}\label{exp-Pb}
\Pb(B) &=& \frac{1}{|B|^2} \Big(\sum_{x\in \Ann(B)}|\Cb_{B}(x)|+\sum_{x\in B-Ann(B)}|\Cb_{B}(x)|\Big) \nonumber \\
&=& \frac{|\Ann(B)|}{|B|} + \frac{1}{|B|^2} \sum_{x\in B-\Ann(B)}|\Cb_{B}(x)|.
\end{eqnarray}

One baby observation is that if  $\Pb(B)  \ne  1$, then $B$ must admit a non-trivial proper sub-skew brace, which follows from the fact that a skew brace having no  non-trivial proper sub-skew brace is always a trivial brace of prime  order (\cite[Theorem A]{BEJP24}). The converse, of course, is not true as $\Pb(B) = 1$ for all trivial  braces. As a direct consequence of the definition, one can easily see that
$$\Pb(B) \leqslant \min\{\Pr(B, +), \Pr(B, \circ)\},$$
where $\Pr(G)$ denotes the commuting probability of a given group $G$.

 We remark that several proofs in this section can also be obtained as  special cases of proofs from Section 5. We have chosen to present these here for completeness.   As we observed above, $\Pb(B)$ lies in the interval $(0, 1]$. But we can say better on these bounds.
\begin{pro}
   For any finite skew brace $B$ with $|B| > 1$, and proper sub-skew brace $H$, the following assertions hold:
   
\text{(1)} $\frac{2}{|B|} \leqslant \Pb(B)$.   \label{pro: card B and comm B}

\text{(2)} $\Pb(B)   < 2\frac{|\Cb_{B}(x)|}{|B|}$ for some $x \in B - \Ann(B)$.

\text(3) $\frac{\Pb(H)}{{[B:H]}^2} < \Pb(B)$. \label{pro:sub skew and commutative}

\noindent Moreover, if $H$ is a trivial brace, then 
  $\frac{1}{{[B:H]}^2}  < \Pb(B)$.
\end{pro}
\begin{proof}
By the definition of $\Pb(B)$, we see that the pairs $(b, 1)$ and $(1, b)$, $b \in B$, contribute to the numerator of $\Pb(B)$. Since these count $2 |B|$, assertion (1) holds.

For assertion (2), we choose an element $x \in B - \Ann(B)$ such that $\Cb_B(x)$ is of maximal order. Then by \eqref{exp-Pb} we directly get
$$|B|^2\Pb(B) \leqslant  |B||Ann(B)| + |B||\Cb_B(x)| - |Ann(B)||\Cb_B(x)|,$$
which further implies
\begin{align*}
        \Pb(B) &\leqslant \frac{|\Ann(B)|}{|B|} + \frac{|\Cb_B(x)|}{|B|}-\frac{|\Ann(B)||\Cb_B(x)|}{|B|^2}\\
        &< \frac{2|\Cb_B(x)|}{|B|}.
      \end{align*}

Assertion (3) follows from
      $$\Pb(B)=~\frac{\sum_{x\in B}|\Cb_{B}(x)|}{|B|^2} > ~ \frac{\sum_{x\in H}|\Cb_{H}(x)|}{|B|^2}=~ \frac{|H|^2\sum_{x\in H}|\Cb_{H}(x)|}{|B|^2|H|^2}=~ \frac{\Pb(H)}{{[B:H]}^2}.$$   
If $H$ is a trivial brace, then $\Pb(H) = 1$. Final assertion now follows from the preceding bound, and the proof is complete.
\end{proof}

For the rest of this article, for a given finite skew brace $B$, let $d$ denote the integer  $|B/\Ann(B)|$, the index of $\Ann(B)$ in $B$.

\begin{pro}\label{thm:upper and lower}
Let $B$ be a finite skew brace.  Then $\frac{2d-1}{d^2}\leqslant \Pb(B)\leqslant \frac{d+1}{2d}$.
\end{pro}
\begin{proof}
We begin by noting that for all $x \in B$, $\Ann(B) \subseteq \Cb_B(x)$. Now using  \eqref{exp-Pb}, we get
      \begin{align*}
            |B|^2\Pb(B) &=|\Ann(B)||B|+\sum_{x\in B-\Ann(B)}|\Cb_{B}(x)|\\
                        &\geqslant |\Ann(B)||B|+|B-\Ann(B)||\Ann(B)|\\
                        &=2|\Ann(B)||B|-|\Ann(B)|^2.\\
                        &=2|B|\frac{|B|}{d}-\frac{|B|^2}{d^2}.\\
                        &=|B|^2\frac{2d-1}{d^2}.
      \end{align*}
     Hence  $\frac{2d-1}{d^2} \leqslant \Pb(B)$.

     Now we prove the upper bound. By Proposition \ref{pro:subgroup} we have $|\Cb_B(x)| \leqslant  \frac{|B|}{2}$ for all $x \in B - \Ann(B)$. Thus, from the definition of the commuting probability, we get
     \begin{align*}
            |B|^2\Pb(B)&\leqslant|\Ann(B)||B|+|(B-\Ann(B))|\frac{|B|}{2}.\\
                        &= |B|\frac{|\Ann(B)|}{2}+\frac{|B|^2}{2}.\\
                        &= |B|\frac{|B|}{2d}+\frac{|B|^2}{2}.\\
                        &= |B|^2\frac{d+1}{2d},
      \end{align*}                  
      which implies that  $\Pb(B)\leqslant \frac{d+1}{2d}$. The proof is complete.
\end{proof}

A natural question which arises here is that whether these bounds are really attained for some skew braces.  We now show that indeed these are attained. The lower bound is attained when $d=p$, a prime integer.

\begin{pro}\label{thm:  prime fixed}
Let $B$ be a finite skew brace and  $p$ be the smallest prime which  divides $|B|$. Then  the following holds true:
\begin{enumerate}
\item  If $|B/\Ann(B)| = p$, then $\Pb(B)=\frac{2p-1}{p^2}$.
\item  $\Pb(B)\leqslant \frac{p+d-1}{pd} \leqslant \frac{2p-1}{p^2}$. Furthermore, if $\Ann(B) \subsetneq \Cb_{(B)}(x)$ for all $x \in B$, then $\Pb(B) \geqslant \frac{pd+d-p}{d^2}$.   
\item  If $B/\Ann(B)$ is of non-prime power order  and $q$ is the second smallest prime divisor of $|B|$, then $\Pb(B) \leqslant \frac{1}{p} + \frac{|\Ann(B)|(p - 1) - 1}{p |B|} + \frac{1}{s|B|}$, where $s$ either $q$ or $p^2$.
\item  In particular, if $|B| = p^n$, then  either $\Pb(B) =1$,  $\Pb(B) = \frac{2p-1}{p^2}$  or $\Pb(B) \leqslant  \frac{p^2+p-1}{p^3} $. 
\item If  $|B| = p^n$,  $|\Ann(B)| = 1$ and $(B, \circ)$ is not an elementary abelian group,  then $\Pb(B) \leqslant \frac{1}{p} + \frac{(p-1)^2}{p^{n+2}}$.
\end{enumerate}
\end{pro}

\begin{proof}
Let $|B/\Ann(B)| = p$. Note that $\Cb_B(x) = \Ann(B)$ for all $x \in B - \Ann(B)$. Then a straightforward computation, using definition  \eqref{exp-Pb}, gives 
$\Pb(B)=\frac{2p-1}{p^2}$, which proves (1). For (2), let $|B/\Ann(B)| = d$. Then, by  \eqref{exp-Pb}, we get
      \begin{align*}
            |B|^2\Pb(B)&\leqslant |\Ann(B)||B|+|(B-\Ann(B))|\frac{|B|}{p}\\
                        &=|\Ann(B)||B|(1-\frac{1}{p}) + \frac{|B|^2}{p}\\
                        &=|B|\frac{|B|}{d}\frac{(p-1)}{p} +\frac{|B|^2}{p},
      \end{align*}
      which implies that $\Pb(B)\leqslant \frac{p+d-1}{pd}$. Since $p \leqslant d$, we get $\frac{p+d-1}{pd} \leqslant \frac{2p-1}{p^2}$.
      
    Now assume that   $\Ann(B) \subsetneq \Cb_{(B)}(x)$ for all $x\in B$.
    Since $\Cb_{(B)}(x)$ is a subgroup of $(B,\circ)$ (by Proposition \ref{pro:subgroup}), it follows that $|\Cb_{(B)}(x)/\Ann(B)| \geqslant p$.  So for any $x\in B$, we get 
      $$|\Cb_{(B)}(x)|\geqslant \frac{p|B|}{d}.$$
A straightforward computation, using the preceding inequality in \eqref{exp-Pb} gives 
 $\Pb(B)\geqslant \frac{pd+d-p}{d^2}$.  This completes the proof of (2).
 
Note that any subgroup $H$ of $(B, \circ)$ of index $p$ is normal in $(B, \circ)$, and therefore, $H$ contains all Sylow $q'$-subgroups of $(B, \circ)$, where $q'$ is a prime divisor of $|B|$ different from $p$.   For (3), if $[(B, \circ) : \Cb_B(x)] = p$ for all $x \in B - \Ann(B)$, then Sylow $q'$-subgroups   of $(B, \circ)$ are contained in $\Ann(B)$, and therefore $B/\Ann(B)$ is of $p$ power order, which contradicts the  given hypothesis. Hence,  there exists at least one element $x \in B - \Ann(B)$ such that $|\Cb_B(x)| \le  \frac{|B|}{s}$, where $s = q$ if $p^2 > q$, otherwise $s = p^2$. Now the assertion (3) directly follows from the definition of the commuting probability of a skew brace.
 
 Next assume that $|B| = p^n$. If $B = \Ann(B)$, the obviously $\Pb(B) =1$, and if $|B/\Ann(B)| = p$, then by (1) we get $\Pb(B) = \frac{2p-1}{p^2}$. So assume that $|B/\Ann(B)| = p^r \geqslant p^2$.  Then by the first assertion of (2) we get
 $$\Pb(B)\leqslant \frac{p+d-1}{pd} =  \frac{p+p^r-1}{p^{r+1}}  \leqslant \frac{p^2+p-1}{p^3},$$
 which proves (4).
 
We now prove (5).  Note that if $\Cb_B(x)$ is a maximal subgroup of $(B, \circ)$ for all $x \in B - \Ann(B)$, then the Frattini subgroup of $(B, \circ)$ is contained in $\Ann(B)$. Since  $(B, \circ)$ is not elementary abelian, its Frattini subgroup is non-trivial, which contradicts the hypothesis that  $|\Ann(B)| = 1$. Hence, there exists an element $y \in B$ such that $\Cb_B(y)$ is not a maximal subgroup of $(B, \circ)$. We now get 
$$ |B|^2\Pb(B)\leqslant |B|+(|B|-2) \frac{|B|}{p} +  \frac{|B|}{p^2}.$$
 The  proof is now complete by taking $|B| = p^n$.
 \end{proof}
 
 We remarks that assertion (1) of the preceding result holds true for any prime divisor of $|B|$.  
 
     Let $(B, +, \circ)$ be a finite skew brace such that $(B,+)$ and $(B, \circ)$ are nilpotent groups. 
     As we know that a finite nilpotent group can be written as a  direct product of its sylow $p$-subgroups, $(B, +)$ can be written as
      $$(B, +) \cong \Pi_{p_i \in \pi(B)}P_i,$$
where $P_i$ is a Sylow $p_i$-subgroup of $(B, +)$. 
     It follows from  \cite[Corollary 4.3]{CSV19} that each $P_i$ is an ideal of $B$, and the above isomorphism is indeed a brace isomorphism. Since $\Pb(H \times K) = \Pb(H) \Pb(K)$ for any two skew braces $H$ and $K$,  $\Pb(B) = \prod_{i=1}^{r}\Pb(P_i)$, where each Sylow $p_i$-subgroup $P_i$ is such that $\Ann(P_i) \subsetneq P_i$. Now by the preceding result we can compute an upper bound for $\Pb(B)$ for a skew brace with nilpotent $(B, +)$ and $(B, \circ)$ (as groups). In particular, this construction can be applied to nilpotent skew braces.

We now explore  a sort of converse of one of the last results.
\begin{pro}
Let $B$ be a skew brace such that $B \neq \Ann(B)$ and $p$ be a prime with $\Pb(B)= \frac{2p-1}{p^2}$. Then $p||B|$. Moreover, if $p$ is the smallest prime  dividing $|B|$, then $B/\Ann(B)\cong \mathbb{Z}_p$.
\end{pro}
\begin{proof}
      By the definition of $\Pb(B)$, the given hypothesis gives the equation
    $$|B|^2 \frac{2p-1}{p^2} =\sum_{x\in B}|\Cb_{B}(x)|.$$
    This is possible only when $p$ divides $|B|$.

 Now assume that $p$ is the smallest prime dividing $|B|$. By \eqref{exp-Pb} we get
 $$|B|^2\Pb(B)\leqslant|\Ann(B)||B|+|(B-\Ann(B))|\frac{|B|}{p}.$$
   This gives
         $$|B|(2p-1)\leqslant |\Ann(B)|p(p-1) + |B|p,$$
          which implies that 
            $$|B|(p-1)\leqslant |\Ann(B)|p(p-1).$$
Hence $|B|/|\Ann(B)|= p$,  as $B \neq \Ann(B)$. Since there exists only one skew  brace of order $p$, we must have   $B/\Ann(B)\cong \mathbb{Z}_p$, which completes the proof.      
\end{proof}

Next we give an example of a skew brace where the upper bound of  Proposition \ref{thm:upper and lower} is attained.

\begin{exa}\label{3/4}
      Let $B=\mathbb{Z}_4$, consider `$+$' as natural `$+$' and `$\circ$' is defined by
      $$x\circ y=x+y+2xy ~\mbox{ for all } x,y \in B,$$ 
      where $xy$ is the natural multiplication. 
      It is easy to see that $(B, +, \circ)$ is a skew brace. Let $\mathbb{Z}_4= \{0, 1, 2, 3\}$. Now, notice that 
      $\Cb_{(B)}(0)=\mathbb{Z}_4$, $\Cb_{(B)}(1)=\{0,2\}$, $\Cb_{(B)}(2)=\mathbb{Z}_4$, $\Cb_{(B)}(3)=\{0,2\}$.
      So $d=2$, and therefore 
      $$\Pb(B)=\frac{4+2+4+2}{4^2}=\frac{3}{4}=\frac{d+1}{2d}$$
      as $|B|=4$
\end{exa}
We now compute the commuting probability of a cyclic brace whose multiplicative group is commutative. Let $(B,+,\circ)$ be a finite cyclic brace whose multiplicative group is commutative. Then, by  \cite[Theorem 1]{WR07}, $B=\mathbb{Z}_n$ with `$+$' as natural `$+$' and `$\circ$' is defined by
$$x\circ y=x+y+dxy ~\mbox{ for all } x,y \in B,$$ 
where $xy$ is the natural multiplication such that $p|d|n$ for each prime divisor $p$ of $n$. 

\begin{thm}
    Let $B=\mathbb{Z}_n$, concider `$+$' as natural `$+$' and `$\circ$' is defined by
    $$x\circ y=x+y+dxy ~\mbox{ for all } x,y \in B,$$ 
    where $xy$ is the natural multiplication such that $p|d|n$ for each prime divisor $p$ of $n$. Then 
    \[\Pb(B)=\frac{1}{n^2}\sum_{{x \in (B, +)} \atop{\overline{dx}=\bar{t}}} gcd(t,n).\]
\end{thm}
\begin{proof}
    Let $\bar{x}\in B$, then $\gamma^+(\bar{x},\bar{y})=1$ and $\gamma^{\circ}(\bar{x},\bar{y})=1$ for all $\bar{y}\in B$. Now $\bar{y}\in \Cb_B(x)$ if and only if $\bar{x}\ast\bar{y}=1$ if and only if $d\bar{x}\bar{y}=1$. Set $\bar{t} = \overline{dx}$. Then  $\bar{y}\in \Cb_B(x)$ if and only if  the order of $\bar{t}$ in $(\mathbb{Z}_n,+)$ divides $y$.  Now 
    \begin{align*}
        |\Cb_B{(x)}|&=|\{\bar{y}\in B \mid y = m |\bar{t}|  \text{ for some } m\in\mathbb{Z} \}|\\
                  &=|\{0\leqslant m|\bar{t}|\leqslant n-1 \mid  m \in \mathbb{Z}\}\\
                  &=|\{0\leqslant m\leqslant \frac{n-1}{|\bar{t}|}  \mid  m \in \mathbb{Z} \}|  \\
                  &=|\{0\leqslant m\leqslant \frac{n}{|\bar{t}|}-1  \mid  m \in \mathbb{Z}\}|  \\
                  &= \frac{n}{|\bar{t}|}  = gcd(t, n).
    \end{align*}
   Hence, by the definition of the commuting probability, we get
    \[\Pb(B)=\frac{1}{n^2}\sum_{{x \in (B, +)} \atop{\overline{dx}=\bar{t}}} gcd(t,n),\]
    which completes the proof.
\end{proof}

Now we obtain exact upper bounds for $\Pb(B)$ for all finite skew braces  $B$.
\begin{thm}  \label{thm: (5/8,1)}
 Let $(B,+, \circ)$ be a finite skew brace. Then $\Pb(B) = 1$,  $\Pb(B) = \frac{3}{4}$ or $\Pb(B) \leqslant \frac{5}{8}$.
\end{thm}
\begin{proof}
Let  $B$ be any finite skew brace. Let $f:\mathbb{R}_{> 0}\rightarrow \mathbb{R}_{> 0}$ be a function defined by $f(x)=\frac{x+1}{2x}$. Then the derivative of $f$ is 
$$f'(x)= \frac{-1}{x^2} ~ (\leqslant 0).$$
Thus $f$ is a monotone decreasing function.  If $d=1$, that is, $B$ is a trivial skew brace of abelian-type, then $\Pb(B) = 1$. If $d = 2, 3$, then, by Proposition \ref{thm: prime fixed}, we get $\Pb(B)=\frac{3}{4}$ and  $\Pb(B)=\frac{5}{9}$, respectively. For all $d \geqslant 4$, we apply Proposition \ref{thm:upper and lower}. Note that for $d=4$, we obtain $\Pb(B)\leqslant \frac{5}{8}$.  Since $f$ is a monotone decreasing function  and $\frac{5}{9}\leqslant \frac{5}{8}$, it follows that $\Pb(B) \leqslant \frac{5}{8}$ for all $d \geqslant 3$, which completes the proof.    
\end{proof}

We now characterise finite skew braces $B$ such that $\Pb(B) = \frac{5}{8}$.

\begin{lemma}
For a finite skew brace $B$, $\Pb(B) = \frac{5}{8}$ if and only if $|B/\Ann(B)|= 4$ and $|\Cb_B(x)| = \frac{|B|}{2}$ for all $x \in B -\Ann(B)$.
\end{lemma}
\begin{proof}
If $\Pb(B) = \frac{5}{8}$, then, using  Proposition \ref{thm:upper and lower} and Proposition \ref{thm:  prime fixed} (1),  we get $|B/\Ann(B)|= 4$, which is possible only if  $|\Cb_B(x)| = \frac{|B|}{2}$ for all $x \in B -\Ann(B)$. On the other hand if $|B/\Ann(B)|= 4$ and $|\Cb_B(x)| = \frac{|B|}{2}$ for all $x \in B -\Ann(B)$, then it follows, by Proposition \ref{thm:  prime fixed} (2), that $\Pb(B) = \frac{5}{8}$.
\end{proof}
A host of examples satisfying the hypothesis of the preceding result  can be obtained using GAP.  One such skew brace is $(Q_8, +, +_{op})$, where $(Q_8, +)$ is the quaternion group of order 8. We can say more in the following result.

\begin{thm}
   Let $B$ be a finite skew brace with $\Pb(B)=5/8$. Then both the additive and the multiplicative groups of quotient skew brace $B/\Ann(B)$ are isomorphic to $\mathbb{Z}/2\mathbb{Z} \times \mathbb{Z}/2\mathbb{Z}$.
\end{thm}
\begin{proof}
    By the preceding lemma, we know that $|B/\Ann(B)|=4$ and $[(B,\circ) : \Cb_B(x)] = 2$ for all $x \in B-\Ann(B)$. If possible, suppose that the additive group $B/\Ann(B)$ is cyclic, say, generated by $y +\Ann(B)$ for some $y \in B - \Ann(B)$. Then $(B, +)$ is an abelian group. If $y \in \Cb_B(y)$, then it turns out that $B = \Cb_B(y)$, which is not possible. Note that $3y \in \Cb_B(y)$ if and only if $y \in \Cb_B(y)$. Hence, $2y \in \Cb_B(y)$, as $\Ann(B) \subsetneq \Cb_B(y)$. Thus, $y \in \Cb_B(2y)$, which, using the fact that $(B, +)$ is abelian,  implies that $ \Cb_B(2y) = B$, not possible again. Hence, $(B/\Ann(B), +)$ can not be cyclic, and therefore isomorphic to $\mathbb{Z}/2\mathbb{Z} \times \mathbb{Z}/2\mathbb{Z}$. That $(B/\Ann(B), \circ)$ is not cyclic can be proved similarly. This competes the proof. 
 \end{proof}
 
 Unlike  group theory result, the  converse of the preceding result is  not true. There are numerous examples which establish this fact. One can use gap to see that the skew braces $B$ with GAP ids $(8, 45)$, $(8, 46)$, $(24, 664-666)$ are some counter examples to the converse of the preceding result.

We now prove the following interesting results, which we'll use later and whose group theory analogues follow from \cite{PG70}.
\begin{thm}\label{Ideal and cbd}
Let $(B,+, \circ)$ be a  skew  brace of finite order and $N$ be an ideal of $B$. Then $\Pb(B)\leqslant \Pb(N)\Pb(B/N)$. In particular, $\Pb(B)\leqslant \Pb(B/N)$ and $\Pb(B) \leqslant \Pb(N)$.
\end{thm}
\begin{proof}
We work with group $(B, \circ)$ and the quotient group $(B, \circ)/(N, \circ)$. For the simplicity of notation we suppress the use of `$\circ$' and write $x \circ N = xN$ for a left coset of $N$ in $(B, \circ)$. For the product of two subgroups $N_1$ and $N_2$  of $(B, \circ)$ we simply write $N_1 N_2$.

 We first prove that $({\Cb_{B}(x)  N)/N} \subset \Cb_{B/N}(x  N)$ for all $x\in B$. Let $y \in \Cb_{B}(x)$.  Since $N$ is an ideal of $B$, we know that 
  $x + N = x \circ N$ for all $x \in B$. Thus by an easy computation we get 
 \begin{eqnarray*}
 (y  N) * (x  N) &=& (y*x)  N  = N,\\
  (x  N) * (y  N) &=& (x*y)  N  = N
  \end{eqnarray*}
  and 
  $$[x N, y  N]^{\circ} = [x, y]^{\circ}  N  = N.$$
Hence $({\Cb_{B}(x)  N)/N} \subset \Cb_{B/N}(x N)$.

Note that $N \cap \Cb_B(x) = \Cb_N(x)$ for all $x \in B$. Thus, as $\Cb_B(x)$ and  $\Cb_N(x)$ are subgroups of $(B, \circ)$, we get
$$(\Cb_B(x)  N)/N \cong \Cb_B(x)/(\Cb_B(x) \cap N) = \Cb_B(x)/\Cb_N(x).$$
Now, using the definition of $\Pb(B)$, we have
\begin{align*}
  |B|^2\Pb(B) &=\sum_{x\in B}|\Cb_{B}(x)| =\sum_{b N\in{B/N}} ~ \sum_{x\in b N}|\Cb_{B}(x)|\\
             &=\sum_{ b N\in{B/N}} ~ \sum_{x\in b N}|{(\Cb_{B}(x) N})/N||\Cb_N(x)|\\
             &\leqslant \sum_{b N\in{B/N}} ~ \sum_{x \in bN}|\Cb_{B/N}(x N)| |\Cb_N(x)|\\  
             &= \sum_{b N\in{B/N}}\Big(|\Cb_{B/N}(b N)|\sum_{x\in b N}|\Cb_N(x)|\Big)\\   
             &= \sum_{b N\in{B/N}}\Big(|\Cb_{B/N}(b N)|  \sum_{y\in N}|\Cb_{B}(y)\cap (b N)|\Big).\\                      
      \end{align*}
      Let $z \in   \Cb_B(y) \cap (b N)$.  Then $z N = b N$ as $z \in b  N$. Now, since $\Cb_{B}(y)$ is a subgroup of $(B, \circ)$ and $z \in \Cb_B(y)$, we get
      $$\Cb_{B}(y)\cap (z N) = (z \Cb_{B}(y)) \cap (z N) = z (\Cb_{B}(y)\cap N) = z   \Cb_N(y),$$
     which gives $|\Cb_{B}(y)\cap(b N)| = |\Cb_N(y)|$. By clubbing in these values in the preceding inequality, we get 
    \begin{align*}
        |B|^2\Pb(B) &\leqslant \sum_{bN\in B/N} \Big(|\Cb_{B/N}(b  N)| \sum_{y\in N}|\Cb_{N}(y)|\Big)\\
        &= \Big(\sum_{bN\in B/N} |\Cb_{B/N}(b  N)| \Big) \Big(\sum_{y\in N}|\Cb_{N}(y)|\Big),
    \end{align*}
      which, by the definition of the commuting probability,  can be written as
     $$|B|^2\Pb(B)\leqslant \big(|B/N|^2\Pb(B/N)\big)\big(|N|^2\Pb(N)\big).$$   
     Hence   $\Pb(B)\leqslant \Pb(N)\Pb(B/N)$ and the proof is complete. 
\end{proof}

We conclude this section by generalising one of the assertions of the preceding result to sub-skew braces.
\begin{pro}
    Let $H$ be a sub-skew brace of a finite skew brace $B$. Then $\Pb(B) \leqslant \Pb(H)$.
\end{pro}
\begin{proof}
  For any $x \in B$, note that  $\Cb_H(x) \subseteq \Cb_B(x)$ are subgroups of $(B, \circ)$.  Let  $x \in B$ such  $[\Cb_B(x) :\Cb_H(x)] = k$ for some integer $k$, and  $\{b_i\}_{i=1}^{k}$ be a set of left coset representatives  of $\Cb_H(x)$ in $\Cb_B(x)$ as subgroup in $(B, \circ)$. It follows that $b_l\circ H \ne  b_m \circ H$ for $l, m \in \{1, \ldots,k\}$ with $l\neq m$. For, if  $b_l\circ H=b_m\circ H$, then ${b_l}^{-1} \circ b_m\in H \cap \Cb_B(x) = \Cb_H(x)$, a contradiction. Hence
    $[(\Cb_B(x),\circ),(\Cb_H(x): \circ)]\leqslant [(B,\circ):(H,\circ)]$, which gives
    $$|\Cb_B(x)|=[\Cb_B(x) :  \Cb_H(x)] |\Cb_H(x)| \leqslant [(B,\circ):(H,\circ)] |\Cb_H(x)|.$$
   Now
    \begin{align*}
        \sum_{x\in B}|\Cb_B(x)| & \leqslant [(B,\circ):(H,\circ)]\sum_{x\in B}|\Cb_H(x)|\\
        &= [(B,\circ):(H,\circ)]\sum_{y\in H}|\Cb_B(y)|\\
        & \leqslant {[(B,\circ):(H,\circ)]}^2\sum_{y\in H}|\Cb_H(y)|.
    \end{align*}
 Using this, we finally get
        $$\Pb(B)=\frac{\sum_{x\in B}|\Cb_B(x)|}{|B|^2}\leqslant \frac{{[(B,\circ):(H,\circ)]}^2\sum_{y\in H}|\Cb_H(y)|}{|B|^2}=\Pb(H),$$
        which completes the proof.
\end{proof}


\section{Nilpotency, Isoclinism and Commuting Probability}

We first study the impact of the commuting probability on the nilpotency of skew braces. As we have seen above, if $\Pb(B) > \frac{3}{4}$, then $B$ is a trivial brace. We here show that a skew brace $B$ is nilpotent if  $\Pb(B) > \frac{65}{128}$.  We start with a very basic observation on a skew brace $B$ with $|\Gamma_2(B)| = 2$. 

\begin{pro}
 Let $(B, +, \circ)$ be a skew brace such that  $|\Gamma_2(B)| = 2$. Then $B$ is a  nilpotent skew brace of nilpotency class $2$.
\end{pro}
\begin{proof}
   Let $\Gamma_2(B) = \{1, b\}$. Then the set $\{a*b, b*a, a+b-a-b \mid a, \in B\}$  generates $\Gamma_3(B)$.  An easy computation shows that all elements of the generating set of $\Gamma_3(B)$ must be the identity element, proving that $\Gamma_3(B) =\{1\}$. This completes the proof.
\end{proof}


\begin{lem} \label{two-sided pn centrally nilpotent}  
    Let $(B, +, \circ)$ be a two-sided skew brace of order $p^n$. Then $B$ is a nilpotent skew brace.
\end{lem}
\begin{proof}
    It follows from \cite[Proposition 3.22]{LV23} that $\Ann(B)$ is a non-trivial ideal of $B$. Hence, it easily follows by an inductive argument that $B$ is a nilpotent skew brace.
\end{proof}

As we observed above, every skew brace whose multiplicative group is abelian is two-sided. Thus as an immediate consequence of the preceding result we get
\begin{cor}\label{p2 centrally nilpotent}
Any skew brace of order $p^2$ is nilpotent.
\end{cor}

There do exist skew braces of prime power order which are not nilpotent. One such example of order $p^3$ for $p=2$  is  $(8,18)$ in GAP library (see \cite{GAP, VK22}). There are total 5 skew braces of order 8, namely (8, 2), (8, 3), (8, 18), (8, 19), (8, 47),  which are not two-sided. It turns out that $Pb(B) \le \frac{1}{2}$ for all skew braces $B$ mentioned in the preceding sentence.  For more knowledge on non-nilpotent braces of prime power order see \cite{SM25, DP24}.

\begin{cor}
Let $(B,+, \circ)$ be a skew brace of order $p^2$. Then $\Pb(B)=1$ or $\Pb(B) =\frac{2p-1}{p^2}$.
\end{cor}
\begin{proof}
    By Corollary \ref{p2 centrally nilpotent}, $B$ is a nilpotent skew brace. So $\Ann(B)\neq\{1\}$, which implies that  $|\Ann(B)|$ is $p$ or $p^2$.  If $|\Ann(B)|=p^2$, then $\Pb(B)=1$. If $|\Ann(B)|=p$, then by Proposition \ref{thm: prime fixed}, $\Pb(B)=\frac{2p-1}{p^2}$. 
\end{proof}

 We  remark that a finite skew brace $B$ is nilpotent if and only of $B/\Ann(B)$ is nilpotent. 
If $B$ is a finite skew brace which is not nilpotent, then there exists a positive  integer $m$ such that $\Ann(B/\Ann_m(B)) = 1$. By Theorem \ref{Ideal and cbd} we know that $\Pb(B) \leqslant \Pb(B/\Ann_m(B))$.

 We are now ready to prove
 
 \begin{thm}   \label{13/14 nilpotent}
    If B is a finite skew brace such that $\Pb(B) > \frac{65}{128}$, then $B$ is nilpotent.
\end{thm}
\begin{proof}
 Since $\Pb(B)> \frac{65}{128}$, it follows from Proposition \ref{thm:upper and lower} that $d \in \{1, 2, \ldots, 63\}$. A GAP search for non-nilpotent skew braces of orders lying in the set $\{1, \ldots, 63\} - \{32, 48\}$ shows that $\Pb(B) \le \frac{1}{2}$. As remarked above, a finite skew brace $B$ is nilpotent if and only of $B/\Ann(B)$ is nilpotent. So assume that $B$ is a non-nilpotent skew brace such that $B/\Ann(B)$ is of order 32 or 48. Let $|B/\Ann(B)| = 32$. Then $(B, \circ)$ can not be abelian as, otherwise, $B$ will be a two sided skew brace, and therefore nilpotent.  If $|\Ann(B/\Ann(B))| = 1$, then, by Proposition \ref{thm:  prime fixed} (5), $\Pb(B/\Ann(B)) \leqslant  \frac{1}{2} + \frac{1}{128} = \frac{65}{128}$. If $|\Ann(B)| > 1$, then $\Pb(B) \le \Pb(B/\Ann(B)) \leqslant \frac{1}{2}$. Now take  $|B/\Ann(B)| = 48$. If $|\Ann(B/\Ann(B))| = 1$, then, by Proposition \ref{thm:  prime fixed} (3), $\Pb(B/\Ann(B)) \leqslant  \frac{73}{144} < \frac{65}{128}$. If  $|\Ann(B)| > 1$, then, again,  $\Pb(B) \le \Pb(B/\Ann(B)) \leqslant \frac{1}{2}$. The proof is complete.
\end{proof}

We now turn attention to isoclinism of skew braces and prove that the commuting probability of skew braces is invariant under isoclinism.
We begin  by defining the concept of isoclinism of skew  braces introduced in \cite{LV23}  (The concept of isoclinism of groups was introduced by P. Hall \cite{Hall40}). For a skew  brace $B$, it is easy to check that the following maps are well defined:
\begin{align*}
\phi^B_{+}&:(B/{\Ann B})^2\to \Gamma_2(B), \hspace{1.5cm}(\bar{a},\bar{b})\mapsto [a,b]^{+},\\
\phi^B_\ast&:(B/{\Ann B})^2\to \Gamma_2(B),\hspace{1.5cm}(\bar{a},\bar{b})\mapsto a\ast b,
\end{align*}
where $\bar{a}, \bar{b} \in B/\Ann(B)$. 
We say that two skew  braces $A$ and $B$ are {\em isoclinic} if there exist  brace isomorphisms $\xi : A/{\Ann A} \to B/{\Ann B}$ and $\theta : \Gamma_2(A) \to \Gamma_2(B)$ such that the diagram
\begin{equation}\label{dia:isoclinism}
\begin{tikzcd}
\Gamma_2(A) \arrow[d,"\theta"] &(A/{\Ann A})^2\arrow[l,"\phi_{+}^A"']\arrow[r,"\phi_\ast^A"]\arrow[d,"\xi \times \xi"] &\Gamma_2(A) \arrow[d,"\theta"]\\
\Gamma_2(B) &(B/{\Ann B})^2\arrow[l,"\phi_{+}^B"]\arrow[r,"\phi_\ast^B"'] & \Gamma_2(B)
\end{tikzcd}
\end{equation}
commutes. The pair $(\xi,\theta)$ is called a  {\em skew  brace  isoclinism}. It is not difficult to see that isoclinism is an equivalence relation in the category of all skew left braces.  A skew brace $B$ is said to be a {\it stem skew brace} is $ \Ann(B)  \subseteq \Gamma_2(B) $. That each isoclinism class (or family) admits a stem skew brace is proved in \cite[Theorem 2.18]{LV23}.

 We now prove that commuting probability is invariant under isoclinism of finite skew braces.   An analogous result for finite groups was proved in \cite{PL95}.
\begin{thm}\label{thm: isoclinism invariant}
Let $A$ and $B$ be two isoclinic skew braces of finite order. Then
$\Pb(A)=\Pb(B)$
and
$|A^\prime\cap \Ann(A)|=|B^\prime\cap \Ann(B)|$.
\end{thm}
\begin{proof} 
Let $(\xi,\theta)$ be an isoclinism between $A$ and $B$. Since $A/\Ann(A) \cong B/\Ann(B)$, for the first assertion we only need to prove that $|A/\Ann(A)|^2\Pb(A) =|B/\Ann(B)|^2\Pb(B).$
First we expand the right hand side. 
      \begin{align*}
            |B/\Ann(B)|^2\Pb(B) &=\frac{1}{|\Ann(B)|^2}|\{(x, y) \in B \mid [x, y]^+ = x\ast y = y\ast x =1\}|\\
                               &=\frac{1}{|\Ann(B)|^2}|\{(x, y) \in B \mid \phi_+^B(\bar{x},\bar{y})=\phi_\ast ^B(\bar{x},\bar{y})=\phi_\ast ^B(\bar{y},\bar{x})=1\}|\\
                               &=|\{(\bar{x},\bar{y}) \in B/\Ann(B)\mid \phi_+^B(\bar{x},\bar{y})=\phi_\ast ^B(\bar{x},\bar{y})=\phi_\ast ^B(\bar{y},\bar{x})=1\}|.
      \end{align*}

      Similarly,
      $$|A/\Ann(A)|^2\Pb(A)=|\{(\bar{x}, \bar{y}) \in A/\Ann(A)\mid \phi_+^A(\bar{x},\bar{y})=\phi_\ast ^A(\bar{x},\bar{y})=\phi_\ast ^A(\bar{y},\bar{x})=1\}|.$$
      Now using the fact that $\theta$ is an isomorphism and the diagram \eqref{dia:isoclinism} commutes, we get
      \begin{align*}
            |A/\Ann(A)|^2\Pb(A) &=|\{(\bar{x}, \bar{y}) \in A/\Ann(A):\theta \phi_+^A(\bar{x},\bar{y})=\theta \phi_\ast ^A(\bar{x},\bar{y})=\theta \phi_\ast ^A(\bar{y},\bar{x})=1\}|,\\
                              &=|\{(\bar{x}, \bar{y}) \in A/\Ann(A):\phi_+^B(\xi,\xi)(\bar{x},\bar{y})=\phi_\ast ^B(\xi,\xi)(\bar{x},\bar{y})=\phi_\ast ^B(\xi,\xi)(\bar{y},\bar{x})=1\}|,\\
                              &=|\{(\bar{x}^\prime,\bar{y}^\prime) \in B/\Ann(B):\phi_+^B(\bar{x}^\prime,\bar{y}^\prime)=\phi_\ast ^B(\bar{x}^\prime,\bar{y}^\prime)=\phi_\ast ^B(\bar{y}^\prime,\bar{x}^\prime)=1\}|,
  \end{align*}
                              where $\xi (\bar{y})=\bar{y}^\prime,\xi (\bar{x})=\bar{x}^\prime$.
Since $\xi$ is an isomorphism, from the preceding identities, we get
$$|A/\Ann(A)|^2\Pb(A) =|B/\Ann(B)|^2\Pb(B),$$
which, using the fact that $|A/\Ann(A)| = |B/\Ann(B)|$, proves $\Pb(A) = \Pb(B)$.

Note that $\Gamma_2(A/\Ann(A)) = \Gamma_2(A) \Ann(A)/\Ann(A)$. Now by the third isomorphism theorem of skew braces \cite{BEP24}, we get 
$$\Gamma_2(A/\Ann(A)) = \Gamma_2(A) /\big(\Gamma_2(A) \cap \Ann(A)\big).$$
Similarly $\Gamma_2(B/\Ann(B)) \cong \Gamma_2(B)/\big(\Gamma_2(B)\cap \Ann(B)\big)$.
 The second assertion now follows using the given hypotheses.
\end{proof}

A stem skew brace in a given isoclinism family of skew braces is a skew brace $B$  such that $\Ann(B)\subseteq \Gamma_2(B)$. We now derive some easy consequences of the preceding theorem.
\begin{cor}
  Let $(B,+,\circ)$ be a skew  brace of finite order such that $\Gamma_2(B) \cap \Ann(B)=1$.  Then there exist a finite skew  brace $A$ such that $\Pb(A)=\Pb(B)$ and $\Ann(A)=1$ and $\Gamma_2(A)$ isomorphic to $\Gamma_2(B)$.
\end{cor}
 \begin{proof}
By  \cite[Theorem 2.18]{LV23}, there exists a stem skew brace $A$ such that $A$ is isoclinic to $B$. This $A$ is the required skew brace.
\end{proof}    
    
\begin{cor}
Let $A$ and $B$ be two isoclinic skew  braces of finite order such that B is non-trivial and simple. Then $A$ is isomorphic to $B \times C$, where $C$ is a skew  brace with $\Pb(C)=1$.
\end{cor}
\begin{proof}
  Since $B$ is non-trivial and simple, we have $B = \Gamma_2(B)$ and $\Ann(B) = 1$. Thus $\Gamma_2(A) \cong B$ and $$\Ann(\Gamma_2(A)) \cong \Ann(\Gamma_2(B)) = \Ann(B) = 1.$$
  Note that $\Gamma_2(A) \cap \Ann(A) \subset \Ann(\Gamma_2(A)) = 1$ and $A/\Ann(A) \cong B/\Ann(B) \cong B = \Gamma_2(A)$. Thus, we get an exact sequence
  $$1 \to \Ann(A) \to A \to \Gamma_2(A) \to 1,$$
  which, comparing the orders, gives $A \cong  \Gamma_2(A) \times \Ann(A) \cong B \times \Ann(A)$. Thus, $\Ann(A)$ is the required skew brace $C$. 
\end{proof}

\begin{remark}
We conclude this section with two easy facts whose proofs follow directly from the definition. By \cite[Theorem 2.18]{LV23} we know that each isoclinism class of skew braces admits a stem skew brace. Let a given isoclinism class of skew braces contain a finite skew brace. Then every stem skew brace in that class is finite. It is proved in \cite[Proposition 2.17]{LV23} that any two stem skew braces in an isoclinism class have the same order. We remark that any two skew braces having equal finite orders in an isoclinism class of skew braces are simultaneously stem skew braces. 
\end{remark}

 
\section{COMMUTING PROBABILITY OF INFINITE SKEW BRACES}

The definition of the commuting probability defined for finite skew braces above does not bear any meaning for infinite skew braces. The concept of the commuting probability of an infinite group was also introduced by Gustafson \cite{WHG73} in $1973$, where he proved that the commuting probability of a non-abelian compact Hausdorff topological group  is always bounded above by $\frac{5}{8}$. Further, it was proved in \cite{EKG07} that the commuting probability of a non-abelian compact Hausdorff topological group  is $\frac{5}{8}$ if and only if $G/Z(G)\cong \mathbb{Z}_2\times \mathbb{Z}_2$. For a similar study of the commuting probability of infinite groups and its advancements, the reader is referred to \cite{AMV17, MCHT20}. It will be interesting to obtains brace analog of this theory.  In this section we define this concept for infinite skew braces and obtain a skew brace analogue of the above bound. We now prepare for the definition of the commuting probability of infinite skew braces.

We start by recalling the definition of  a topological group. A group $G$ is said to be  a  \emph{topological group} if $G$ is a topological space such that the maps $G\times G \rightarrow G$,  $(x, y) \mapsto xy$ and $G \rightarrow G$,  $x \mapsto x^{-1}$ are continuous.  We now define a topological skew brace. A skew brace $(B,+, \circ)$ is called a \emph{topological skew brace} if both the groups $(B,+)$ and $(B,\circ)$ are topological groups admitting a common topology. Thus, all four maps $\tau^+_1 :(B, +) \times (B, +) \to (B, +)$, $\tau^{\circ}_1 :(B, \circ) \times (B, \circ) \to (B, \circ)$, $\tau_2^+: (B, +)  \to (B, +)$ and $\tau_2^{\circ} : (B, \circ)  \to (B, \circ)$, respectively, given by $\tau^+_1(a, b) = a+b$, $\tau^{\circ}_1(a, b) = a \circ b$, $\tau^{+}_2(a) = -a$ and $\tau^{\circ}_2(a) =  a^{-1}$ are continuous. Also note that $\Id_B : B \to B$, the identity map, is also continuous. We can construct topological spaces $B^n := B \times \cdots \times B$ (n copies) by considering the product topology on $B^n$.  The concepts of compact and Hausdorff topological skew brace $(B, +, \circ)$ are defined by considering these properties with respect to a given common topology for $(B, +)$ and ($B, \circ)$.  We now make some key observations in the form of lemmas. 

\begin{lem}
Let $(B, +, \circ)$ be a topological skew brace. Then the map $\tau : B \times B \to B$, given by $\tau(a, b) = a * b$, is continuous.    
\end{lem} 
\begin{proof}
 Note that the composition of two continuous maps is again continuous. So the map $f_1: B \times B \to B \times B \times B$, defined by $f(a, b) = (\tau_2^+(a), \tau_1^{\circ}(a, b), \tau_2^+(b))$ is continuous.  Also the map $f_2:  B \times B \times B \to B$, defined by $f_2(a, b, c) = \tau_1^+\big(\tau_1^+(a, b), c\big)$, is continuous. Now, $\tau$, which is nothing but $f_2$ composed with $f_1$, is continuous.
\end{proof}

The following is well known.
\begin{lem}
Let $(B, +, \circ)$ be a topological skew brace. Then the commutator maps $\gamma^+ : (B, +) \times (B, +) \to (B, +)$ and  $\gamma^{\circ} : (B, \circ) \times (B, \circ) \to (B, \circ)$ are continuous.
\end{lem}

Our last key observation is 
\begin{lem}
Let $(B, +, \circ)$ be a Hausdorff topological skew brace. Then $Z:= \{(a, b) \mid \tau(a, b) = \gamma^{\circ}(a, b) = \gamma^+(a, b) = 1\}$ is a closed subset of $B \times B$, where  $B \times B$ is a topological space with the product topology. Moreover, the sets  $\Cb_B(x)$ and $\Ann(B)$ are also closed subsets of $B$. 
\end{lem}
\begin{proof}
The first  assertion follows by noting that $Z = \tau(1) \cap \gamma^{\circ}(1) \cap \gamma^+(1)$. The second assertion also holds by considering appropriate continuous maps.
\end{proof}

    Let $B := (B,+, \circ)$ be a compact Hausdorff topological skew brace. Then $(B,+)$ is a compact Hausdorff topological group, and therefore,  there is a unique probability Haar measure space $((B, +), \mathcal{M}, \mu)$.  Consider the product measure $\mu\times \mu$ on the product space $B\times B$, which turns out to be  a probability measure. Since $Z$ is closed in $B \times B$, $(\mu \times \mu)(Z)$ is well defined. The \emph{commuting probability} of the skew brace $B$, denoted by  $\Pb(B)$, is defined to be
$$\Pb(B) := (\mu\times \mu)(Z).$$

We remark that if $B$ is a finite skew brace, then $B$ is a compact skew brace with discrete topology, 
and so the Haar measure of $B$ is the counting measure. Therefore, 
\[
\Pb(B) = (\mu\times \mu)(Z) = \frac{|Z|}{|B|^2},
\]
which agrees with the definition of the commuting probability of a finite skew brace given in  \eqref{exp-Pb}.

\begin{remark}\label{rem6.4}
Let $B := (B,+, \circ)$ be a compact Hausdorff topological skew brace, and $((B, +), \mathcal{M}, \mu)$ be the probability Haar measure space.  We claim that $((B, \circ), \mathcal{M}, \mu)$  is a probability Haar measure space.  Since ${\lambda_x}$ is a homeomorphism from  $(B,+)$ onto itself, it follows that 
 ${\lambda_x}$ maps Borel sets to Borel sets, that is, $\mathcal{M}$ is mapped onto itself. Note that for any  $D \in \mathcal{M}$ and $x \in B$, we have    $x \circ D =  x + \lambda_x(D)$. So $x\circ D\in \mathcal{M}$.  Define a function $\mu_{\lambda_x}:\mathcal{M}\rightarrow[0,\infty]$ by $\mu_{\lambda_x}(D)=\mu({\lambda_x}(D))$. Since $\lambda_x$ is a group homomorphism on $(B, +)$, it turns out that $\mu_{\lambda_x}$ is also a probability Haar measure on $(B,+)$. Hence,  by the uniqueness of the probability Haar measure, $\mu_{\lambda_x}=\mu$. Now $\mu( x\circ D)=\mu(x+\lambda_x(D))=\mu(\lambda_x(D))=\mu(D) $. So $\mu $ is also a probability Haar measure on $(B,\circ)$. Hence we can simply say that the skew brace $B$ admits the unique probability Haar measure $\mu$.
\end{remark}

For simplicity of notation,  a `compact Hausdorff topological skew brace' will be termed just as a `compact skew brace'. We now prove

\begin{lem} \label{lem: int and cbd}
    Let $(B,+, \circ)$ be a compact  skew brace. Then
$ \Pb(B) \;=\; \int_{B} \mu(\Cb_{B}(x)) \, d\mu(x)$,
where
$\mu(\Cb_{B}(x)) \;=\; \int_{B} \chi_{Z}(x,y)\, d\mu(y)$.
\end{lem}
\begin{proof}
By the definition of the commuting probability, we get 
\begin{align*}
          \Pb(B)&= (\mu \times \mu)(Z) = \int_{B \times B} \chi_{Z} \, d(\mu \times \mu),
\end{align*}
     which  by the Fubini-Tonelli’s Theorem \cite[Section 36]{PRH50} gives 
     \begin{align*}
          \Pb(B) &= \int_{B} \int_{B} \chi_{Z}(x,y) \, d\mu(x) \, d\mu(y) =\int_{B} \mu(\Cb_{B}(x)) \, d\mu(x),
\end{align*}
which, using the given value of  $\mu(\Cb_{B}(x))$,  completes the proof.
\end{proof}

\begin{lem}\label{lem:closed}
    Assume that $H$ is a Borel subgroup of $(B,\circ)$  of a compact  skew brace $(B,+, \circ)$. If $[B : H] \geqslant n$, then  $\mu(H) \leqslant \frac{1}{n}$, where $[B : H]$ denotes the index of $(H,\circ)$ in $(B,\circ)$. Moreover, if $[B : H] = n$, then $\mu(H)= \frac{1}{n}$. Specifically, if $\Ann(B)$ is  properly contained in $B$, then  $\mu(Ann(B))\leqslant\frac{1}{2}$. Further, if the index of $H$ in $(B, \circ)$ is not finite, then  $\mu(H)= 0$.
\end{lem}
\begin{proof}
As $H$ is a Borel subset of $B$,  $H$ is a measurable set, and so $\mu(H)$ is well defined.   
Since $[B : H] \geqslant n$, there exist at least $n$ disjoint cosets 
\[
x_{1}   \circ H, \; x_{2} \circ H, \ldots, x_{n}  \circ H, 
\]
of $H$ in $(B, \circ)$, which are all measurable sets.
Thus, we have
\[
1 = \mu(B) \;\geqslant\; \mu\!\left(\sqcup_{i=1}^{n} x_{i} \circ H \right) 
= \sum_{i=1}^{n} \mu(x_{i}   \circ H) 
= n \mu(H)
\]
If $[B : H] = n$, then $(B, \circ) = \sqcup_{i=1}^{n} x_{i}   \circ H$, and therefore $\mu(H)= \frac{1}{n}$.

Finally assume that the index of $H$ in $B$ is not finite. Contrarily assume that  $\mu(H)  = \epsilon > 0$.  Then there exists a natural number $m$ such that $m \epsilon > 1$. Now choose any $m$ left cosets $\{x_{1} \circ H, \; x_{2} \circ H, \ldots, x_{m} \circ H\}$ of $H$ in $(B, \circ)$. Therefore
\[
1 = \mu(B) \;\geqslant\; \mu\!\left(\sqcup_{i=1}^{m} x_{i} \circ H \right) 
= \sum_{i=1}^{m} \mu(x_{i} \circ H) 
= m \mu(H) = m \epsilon > 1,
\]
which is not possible. The proof is complete. 
\end{proof}

We would like to remark that if $H$ is a Borel subgroup of $(B, +)$, then the assertions of the preceding result hold true. Moreover, if $H$ is a sub-skew brace of $B$, which is a Borel set, then $[(B, +) : H]$ and $[(B, \circ) : H]$ are simultaneously finite or infinite; if finite, then $[(B, +) : H] = [(B, \circ) : H]$.

Let $B:= (B, +, \circ)$ be a compact  skew brace such that $B \ne \Ann(B)$. Suppose that no element $x \in B - \Ann(B)$ has finite centraliser index $[(B, \circ): \Cb_B(x)]$ in $B$. Then 
\begin{align*}
    \Pb(B) &= (\mu \times \mu)(Z)\\
           &= \int_{B} \mu(\Cb_{B}(x)) \, d\mu(x)~~(\text{by Lemma }\ref{lem: int and cbd})\\
           &= \int_{\Ann(B)} \mu(\Cb_{B}(x)) \, d\mu(x) + \int_{B - \Ann(B)} \mu(\Cb_{B}(x)) \, d\mu(x)\\
           & \leqslant \mu(\Ann(B)) +  \int_{B - \Ann(B)} \mu(\Cb_{B}(x)) \, d\mu(x)\\
           & = 0 ~~(\text{by Lemma }\ref{lem:closed}).
 \end{align*}
 Thus we get

\begin{lemma}
For a  compact  skew brace $B$ with $B \ne \Ann(B)$ and  $\Ann(B) = \FCI(B)$,  $\Pb(B) = 0$.
\end{lemma}

Now we prove that $\FCI(B)$ is a Borel set.
\begin{pro}
    Let $(B,+, \circ)$ be a Hausdorff topological skew brace. Then $\FCI(B)$ is a Borel set.
\end{pro}
\begin{proof}
     Let $F_n :=\{y \in B \mid [(B,\circ) : \Cb_B(y)]\leqslant n\}$.  Then $\FCI(B) = \cup_{n\in\mathbb{N}} F_n$. Let $x$ be a limit point of $F_n$. Then there is a net $\{x_\lambda\}$ in $F_n$ which converge to $x$. Since $x_\lambda\in F_n$, we see that  $B=\sqcup_{i=1}^{k}(b_i\circ\Cb_B(x_\lambda) )$ for some $1\leqslant k\leqslant n$ and $b_i\in B$. Let $b\in B$ be an arbitrary element. Then $b\in b_i\circ\Cb_B(x_\lambda)$, for some $1\leqslant i\leqslant k$. So $b_i^{-1}\circ b\in \Cb_B(x_\lambda)$. Define three functions $f_1, f_2, f_3 : B\rightarrow B$ by $f_1(y)=(b_i^{-1}\circ b)\circ y \circ (b_i^{-1}\circ b)^{-1}$ and $f_2(y)=(b_i^{-1}\circ b)\ast y$ and $f_3(y)=y\ast (b_i^{-1}\circ b)$ for all $y\in B$. Then $f_1,f_2,f_3$ are continuous functions. Hence $f_j(x_\lambda)$ converges to $f_j(x)$ for $1\leqslant j \leqslant 3$, and therefore $b_i^{-1}\circ b\in \Cb_B(x)$. This implies that $b\in b_i\circ\Cb_B(x)$. Consequently, $[(B,\circ):\Cb_B(x)]\leqslant k$, which gives  $x \in F_n$ proving $F_n$ is closed in $B$. The proof is complete.  
\end{proof}

We now prove a brace analogue of \cite[Theorem 1.2]{HR12}.
\begin{thm}
    Let $(B,+, \circ)$ be a compact  skew brace. Then the following statements are equivalent:
    \begin{enumerate}
        \item $\Pb(B)>0$ .
        \item $\FCI(B)$ is open in $B$.
        \item $[(B,\circ): \FCI(B)]<\infty$.
    \end{enumerate}
\end{thm}
\begin{proof}
 Set $F:= \FCI(B)$. Since the index of $\Cb_B(x)$ is infinite in $(B, \circ)$ for all $x \in B - F$, it follows that 
 $$\Pb(B)= \int_{B} \mu(\Cb_{B}(x)) \, d \mu(x)= \int_{F} \mu(\Cb_{B}(x)) \, d \mu(x) \leqslant  \int_{F}  \, d \mu(x)=\mu(F).$$ 
 So by  (1), we get  $\mu(F)>0$. Then it follows from  \cite[ Corollary 20.17]{ES63} that $F\circ F^{-1}$ contains a nonempty open set.  Since $F\circ F^{-1} = F$,  by \cite[Theorem 5.5]{ES63}, $F$ is open in $B$.
 
   Now we assume (2).  Since $F$ is  nonempty and open in $B$, we have $\mu(F) > 0$. Then, by Lemma \ref{lem:closed},  $[(B,\circ) : F] < \infty$.
 This proves (3).  If (3) holds, then, by Lemma \ref{lem:closed},  $\mu(F) > 0$. If $\int_{F} \mu(\Cb_{B}(x)) \, d \mu(x) = 0$, then,  by \cite[Theorem 1.39]{WR87}, $\mu \Cb_B = 0$ almost everywhere on $F$. But we know that  $\mu \Cb_B$ is nowhere zero on $F$. Hence,  $\Pb(B)= \int_{F} \mu(\Cb_{B}(x)) \, d \mu(x)>0$, which proves (1), and the proof is complete.
\end{proof}
 
Now we present an example of a compact  non-trivial skew brace $B$ with $\Pb(B)=0$.  Let $(G,+)$ be a non-abelian connected compact Lie group. Let $B := (G, +, +_{op})$. Note that if $[a, x]^+ = 1$, then $a*x=1$. It now follows that  $\Cb_B(x) = \C_{(B, +)}(x) =  \C_{(B, \circ)}(x)$ for all $x \in B$. Thus we get $\Pb(B) = \Pb(G, +, +)$ is equal to the commuting probability of the group $(G, +)$.  Let $\Phi(G, +)$ be the subset of all pairs $(x, y) \in (G, +) \times (G, +)$ such that $x$ and $y$ generate a non-abelian free subgroup of $(G, +)$. Then it follows from  \cite[Proposition 6.86. and Exercise E 6.18.]{HM23} that   Haar measure of the complement  $C$ of $\Phi(G, +)$ in $(G, +)$ is zero. Since $Z$ is contained in $C$, we have $(\mu \times \mu)(Z) = 0$, where $\mu$ is the probability Haar measure. This proves that $\Pb(B) = 0$.

We now get a lower bound  on $\Pb(B)$ for a compact  skew brace $B$.
\begin{pro}\label{pro:infinite and index p}
    Let $(B,+, \circ)$ be a compact  skew brace such that $|B/\Ann(B)|=d$.  Then $\Pb(B)\geqslant \frac{2d-1}{d^2}$. Moreover, if $d = p$, a prime integer,  then $\Pb(B)=\frac{2p-1}{p^2}$. 
\end{pro}
\begin{proof}
By a direct computation, we get
 \begin{align*}
    \Pb(B) &= \int_{B} \mu(\Cb_{B}(x)) \, d \mu(x) ~(\text{by Lemma }\ref{lem: int and cbd})\\
           &= \int_{\Ann(B)} \mu(\Cb_{B}(x)) \, d\mu(x) + \int_{B - \Ann(B)} \mu(\Cb_{B}(x)) \, d\mu(x)\\
           &\geqslant \int_{\Ann(B)}  \mu(\Cb_{B}(x)) \, d\mu(x) + \int_{B - \Ann(B)} \mu(\Ann(B)) \, d\mu(x)\\
           & = \mu(\Ann(B)) + \tfrac{1}{d}\mu(B - \Ann(B))) ~~(\text{by Lemma }\ref{lem:closed})\\
           & = \mu(\Ann(B)) + \tfrac{1}{d}\big(\mu(B) - \mu(\Ann(B))\big)\\
           &= \tfrac{d-1}{d}\mu(\Ann(B)) +\tfrac{1}{d} \mu(B)\\
           &=\frac{d-1}{d^2}+\frac{1}{d} ~~(\text{by Lemma }\ref{lem:closed})\\
           &=\frac{2d-1}{d^2}.   
\end{align*}

Now assume that $d = p$. So,  by Proposition \ref{pro:subgroup}, $\Ann(B)=\Cb_{B}(x)$ for all $x\in B-\Ann(B)$. Let $\{1, y_1, \ldots, y_{p-1}\}$ be a set of coset representatives of $\Ann(B)$ in $B$. Then
\begin{align*}
    \Pb(B)) &= \int_{B} \mu(\Cb_{B}(x)) \, d\mu(x) ~~(\text{ by Lemma }\ref{lem: int and cbd}.)\\
           &= \int_{\Ann(B)} \mu(\Cb_{B}(x)) \, d\mu(x) + \sum_{i=1}^{p-1}\big(\int_{ y_i+\Ann(B)} \mu(\Cb_{B}(x)) \, d\mu(x)\big) \\
           & = \mu(\Ann(B)) + \sum_{i=1}^{p-1} \big(\tfrac{1}{p}\mu(y_i + \Ann(B))\big)\\
           &= \mu(\Ann(B))+\tfrac{p-1}{p}\mu(\Ann(B))\\
           &=\frac{1}{p}+\frac{p-1}{p^2}=\frac{2p-1}{p^2}.           
\end{align*}
The proof is complete.
\end{proof}

We now compute some upper bounds.

\begin{thm}\label{lem: measure of centralizer}
Let $(B,+, \circ)$ be a compact   skew brace. Then the following hold true:
\begin{enumerate}
 \item  For all   $x \in B-\Ann(B)$,  $\mu(\Cb_{B}(x))\leqslant \frac{1}{2}$.
 \item  $\Pb(B)\leqslant \frac{3}{4}$.
 \item If $|B/\Ann(B)|=d$, then  $\Pb(B)\leqslant \frac{d+1}{2d}$.
 \item  In particular, if $|B/\Ann(B)|= d \ge 3$, then $\Pb(B) \leqslant \frac{5}{8}$.
 \item  If $B$ is not  FCI, then $\Pb(B)\leqslant \frac{1}{2}$.
\end{enumerate}
\end{thm}
\begin{proof}
For $x\in B-\Ann(B)$ we know that $\Cb_{B}(x) \subsetneq B$. It now follows from  Proposition \ref{pro:subgroup} that  $\Cb_{B}(x)$ is   a proper subgroup of $(B,\circ)$. Since  $\Cb_{B}(x)$ is   closed, it admits a value under  the Haar measure $\mu$. Hence, by Lemma \ref{lem:closed}, we get   $$ \mu(\Cb_{B}(x))\leqslant \frac{1}{2}.$$  

For the second assertion, we compute
\begin{align*}
    \Pb(B) &= (\mu \times \mu)(Z)\\
           &= \int_{B} \mu(\Cb_{B}(x)) \, d\mu(x)~~(\text{by Lemma }\ref{lem: int and cbd})\\
           &= \int_{\Ann(B)} \mu(\Cb_{B}(x)) \, d\mu(x) + \int_{B - \Ann(B)} \mu(\Cb_{B}(x)) \, d\mu(x)\\
           & \leqslant \mu(\Ann(B)) + \tfrac{1}{2}\mu(B - \Ann(B)))~~(\text{by (1)})\\
           & = \mu(\Ann(B)) + \tfrac{1}{2}\big(\mu(B) - \mu(\Ann(B))\big)\\
           &= \tfrac{1}{2}\mu(\Ann(B)) + \tfrac{1}{2}\mu(B)\\
           &\leqslant \frac{1}{4}+\frac{1}{2} = \frac{3}{4},~~(\text{by Lemma }\ref{lem:closed}).    
\end{align*}
Putting $\mu(\Ann(B)) = \frac{1}{d}$, in the preceding computation, we readily get assertion (3). Assertion (4) holds on the lines of the proof of Theorem \ref{thm: (5/8,1)}, using Proposition \ref{pro:infinite and index p} for $d = 3$, and the last  assertion holds using the fact that $\mu(\Ann(B)) = 0$ (Lemma \ref{lem:closed}) in the above inequalities. The proof is now complete.
\end{proof}

 Recall that $F_n := \{y \in B \mid [(B,\circ) : \Cb_B(y)]\leqslant n\}$ is a Borel set for a compact skew brace $B$.  We consider a specific situation when $B$ a compact skew brace and $|B/\Ann(B)|=4$. Then $F_2 \subseteq F_4$ are Borel sets. Define $E_4 := F_4 - F_2$, which is the collections all elements $x$ of $B$  such that  $[(B,\circ) : \Cb_B(x)] = 4\}$. Note that $E_4 = F_4 \cap (B - F_2)$. Since both $F_4$ and $B-F_2$ are Borel sets, $E_4$ is Borel too, and therefore Haar measure $\mu$ is defined on $E_4$. We can now say little more on the number $\frac{5}{8}$ as a commuting probability.
   \begin{pro}
   Let $B$ be an infinite compact skew brace. Then $\Pb(B) = \frac{5}{8}$ if and only if $|B/\Ann(B)| = 4$ and $\mu(E_4) = 0$.
   \end{pro} 
   \begin{proof}
   Let $\Pb(B) = \frac{5}{8}$. Then, by Proposition \ref{pro:infinite and index p} and Theorem \ref{lem: measure of centralizer}, we obtain $|B/\Ann(B)|=4$. Thus $[(B,\circ) : \Cb_B(x)]$ is either 2 or 4. Now
   \begin{eqnarray*}
   \Pb(B)& =& \int_{\Ann(B)} \mu(\Cb_{B}(x)) \, d\mu(x) + \int_{B - \Ann(B) -E_4} \mu(\Cb_{B}(x)) \, d\mu(x) + \int_{E_4} \mu(\Cb_{B}(x)) \, d\mu(x)\\
   &=& \mu(\Ann(B) + \frac{1}{2}\big(\mu(B) - \mu(\Ann(B)  - \mu(E_4)\big) + \frac{1}{4} \mu(E_4)\\
   & =&  \frac{1}{2} + \frac{1}{2}\big(\mu(\Ann(B))  - \frac{1}{4}\mu(E_4)\big) = \frac{5}{8} - \frac{1}{4}\mu(E_4)).
   \end{eqnarray*}
  This gives $\mu(E_4) = 0$. The converse follows from the preceding equations using the given hypotheses. 
   \end{proof}

We finally get the skew brace analogue of the main result of \cite{WHG73}.  
\begin{thm}
    Let $(B,+, \circ)$ be a compact   skew brace. Then
    $\Pb(B) > \frac{5}{8}$ if and only if  $|B/\Ann(B)| \leqslant 2$. Moreover, $|B/\Ann(B)| = 2$ if and only if  $\Pb(B) = \frac{3}{4}$.
\end{thm}
\begin{proof}
    If $B/\Ann(B)=1$, that is, $B$ is a trivial brace, then $\Pb(B)=1$. So assume that $|B/\Ann(B)| = 2$. Then it follows from   Proposition \ref{pro:infinite and index p} that  $\Pb(B) = \frac{3}{4}$.    Conversely, assume that $\Pb(B) > \frac{5}{8}$. If $\Pb(B) = 1$, then $B=\Ann(B)$. Assume that $B$ is not a trivial brace. Then $|B/\Ann(B)| \geqslant 2$. If $|B/\Ann(B)| = 2$, then, by Proposition \ref{pro:infinite and index p} we get $\Pb(B) = \frac{3}{4}$. So we assume that $|B/\Ann(B)| = d \geqslant 3$. It now follows from  Theorem \ref{lem: measure of centralizer} (4)  that  $\Pb(B) \leqslant \frac{5}{8}$, which proves  the first assertion. The second one follows by using Theorem \ref{lem: measure of centralizer} (3).
 \end{proof}

  We conclude with an explicit example of an infinite skew brace $B$ with $\Pb(B)=\frac{3}{4}$. We first observe that 
 for two  compact skew braces $(B_1,+, \circ)$ and $(B_2,+, \circ)$, it follows that $\Pb(B_1 \times B_2) = \Pb(B_1)\Pb(B_2)$. 
Let $B_1=(\mathbb{R}/\mathbb{Z},+,+)$ be the trivial brace, where $+$ is the natural addition on real numbers  mod $\mathbb{Z}$, and $B_2$ be the skew brace which we defined in Example \ref{3/4}. Then $\Pb(B_1)=1$, $\Pb(B_2)=3/4$, and therefore  $\Pb(B_1\times B_2)=3/4$.\\

\noindent {\bf Acknowledgement:} The second-named author acknowledges the partial support of DST-SERB Grant MTR/2021/000285.


\begin{thebibliography}{100} 

\bibitem{AMV17}
Y.~ Antolin, A.  Martino, and  E. Ventura,  \emph{Degree of commutativity of infinite groups},  Proc. Amer.  Math. Soc. {\bf  145} (2017), 479-485.



\bibitem{BEP24}
A. Ballester-Bolinches, R. Esteban-Romero and V. Perez-Calabuig, \emph{A Jordan-Hölder theorem for skew left braces and their applications to multipermutation solutions of the Yang-Baxter equation}, Proc. Roy. Soc. Edinburgh Sect. A 1{\bf 54} (2024), 793-809.

\bibitem{BEJP24}
A. Ballester-Bolinches, R. Esteban-Romero, P. Jiménez-Seral and V. Pérez-Calabuig,  \emph{Soluble skew left braces and soluble solutions of the Yang–Baxter solution},  Adv. Math. {\bf 455} (2024), Paper No. 109880, 27 pp.


\bibitem{BNY22}
V. G.  Bardakov, M.  V. Neshchadim and Manoj K. Yadav,  \emph{On  $\lambda$-homomorphic skew braces}, J. Pure 	Appl. Algebra {\bf 226} (2022), Paper No. 106961, 37 pp.

\bibitem{BNY23}
V. G.  Bardakov, M.  V. Neshchadim and Manoj K. Yadav,
\emph{Symmetric skew braces and brace systems},
Forum Math. {\bf 35} (2023), 713-738.


\bibitem{BJ23}
M. Bonatto and P. Jedlicka, \emph{Central nilpotency of skew braces}, J Algebra Appl. {\bf 22} (2023), 2350255 (16 pages).

\bibitem{CCD20}
E.~Campedel, A.~ Caranti and I.~Del Corso,  \emph{Hopf-Galois structures on extensions of degree $p^2q$ and skew braces of order $p^2q$: the cyclic Sylow p-subgroup case}, J. Algebra {\bf 556} (2020), 1165-1210.

\bibitem{AC20}
 A. Caranti, \emph{Bi-skew braces and regular subgroups of the holomorph}, J. Algebra {\bf 562} (2020), 647- 665. 
 
 
 

\bibitem{CSV19} 
F. Cedo, A.  Smoktunowicz and L.  Vendramin, {\it Skew left braces of nilpotent type}, Proc. London Math. Soc., {\bf 118} (2019), 1367-1392.

\bibitem{LC19}
L. N. Childs,  \emph{Bi-skew braces and Hopf Galois structures}, New York J. Math. {\bf 25} (2019), 574-588.

\bibitem{CFT25}
I.~Colazzo, M.~ Ferrara and M.~Trombetti, \emph{On derived-indecomposable solutions of the Yang-Baxter equation},  Publ. Mat. {\bf 69} (2025), 171-193.

\bibitem{Drinfeld92}
V. G. Drinfeld, \emph{On some unsolved problems in quantum group theory}, Lecture Notes in Math., {\bf 1510}
Springer-Verlag, Berlin, 1992, 1-8.

\bibitem{ET68}
P.~Erdos and P. ~Turan,  \emph {On some problems of a statistical group-theory. {IV}}, Acta Math. Acad. Sci. Hungar. {\bf 19} (1968), 413-435.

\bibitem{EKG07}
A.~Erfanian, R.~Kamyabi-Gol, \emph{On the mutually commuting {$n$}-tuples in compact groups}, Int. J. Algebra {\bf 1} (2007), 251-262.
    

 \bibitem{PG70}
P.~X.~Gallagher, \emph{ The number of conjugacy classes in a finite group}, Math. Z. {\bf 118} (1970), 175-179.
 
 \bibitem{GAP}
The GAP Group, \emph{Groups Algorithms and Programming,}  version 4.15.1 (2025), available at  http://www.gap-system.org.

\bibitem{GV17}
 L. Guarnieri  and  L. Vendramin, \emph{ Skew braces and the Yang-Baxter equation},  Math. Comp. {\bf 86} (2017), 2519-2534.



\bibitem{WHG73}
W. H. Gustafson,  \emph{What is the probability that two group elements commute?}, The American Mathematical Monthly {\bf 80} (1973),1031-1034.
	
\bibitem{Hall40}	
P. Hall, \emph{The classification of prime power groups}, J. Reine Angew. Math. {\bf 182}  (1940) 130-141.	

\bibitem{PRH50}
P.~R.~Halmos, Measure theory, Van Nostrand, Princeton, N. J. 1950.

\bibitem{ES63}
E.~Hewitt and K.~A.~ Ross, \emph{Abstract harmonic analysis. Vol. I: Structure of topological groups. Integration theory, group representations}, Die Grundlehren der mathematischen Wissenschaften, Band 115
Springer-Verlag, Berlin-Göttingen-Heidelberg; Academic Press, Inc., Publishers, New York, 1963, viii+519 pp.

\bibitem{HM23}
K.~H.~ Hofmann and S.~A.~ Morris, \emph{The structure of compact groups—a primer for the student—a handbook for the expert}, De Gruyter Stud. Math., 25, De Gruyter, Berlin, 2023, xli+1032 pp.

\bibitem{HR12} 
K.~H.~ Hofmann and F.~G.~ Russo, \emph{The probability that  x  and  y  commute in a compact group}, Math. Proc. Cambridge Philos. Soc. {\bf 153} (2012), no. 3, 557-571.

	

\bibitem{JAV23}
E. Jespers,  A.  Van Antwerpen and L.  Vendramin,  \emph{Nilpotency of skew braces and multipermutation solutions of the Yang-Baxter equation}, Commun. Contemp. Math. {\bf 25} (2023), no. 9, Paper No. 2250064, 20 pp.


\bibitem{PL95}
P.~Lescot, \emph{Isoclinism classes and commutativity degrees of finite groups}, J. Algebra {\bf 177} (1995), 847-869.

\bibitem{LV23}
T. Letourmy and L. Vendramin, \emph{Isoclinism of skew braces}, Bull. Lond. Math. Soc. {\bf 55} (2023), no. 6, 2891-2906.

\bibitem{SM25}
S.~Mukherjee,  \emph{Classification of right nilpotent $\mathbb{F}_p$-braces of cardinality $p^5$},  J. Algebra {\bf 665} (2025), 503-537.

 


\bibitem{DP24}
D.~Puljic, \emph{Classification of braces of cardinality $p^4$}, J. Algebra {\bf 660} (2024), 1-33. 

\bibitem{WR87}
W.~Rudin,  \emph{Real and complex analysis},  McGraw-Hill Book Co., New York, 1987, xiv+416 pp.

\bibitem{Rump07}
W. Rump, \emph{Braces, radical rings, and the quantum Yang-Baxter equation},  J. Algebra {\bf 307} (2007), no. 1, 153-170.

\bibitem{WR07}
W.~Rump, \emph{Classification of cyclic braces}, J. Pure Appl. Algebra {\bf 209} (2007), 671-685.


\bibitem{SV18}
A.~Smoktunowicz, L.~Vendramin, \emph{On skew braces (with an appendix by N. Byott and L. Vendramin)}, J. Comb. Algebra {\bf 2} (2018), 47-86.



\bibitem{MCHT20}
M.~ C.~ H.~ Tointon, \emph{Commuting probabilities of  infinite groups},  J. London Math. Soc. (2) {\bf 101} (2020), 1280-1297. 


\bibitem{VK22}
L. Vendramin and  O. Konovalov, {\it Yang-Baxter}, version 0.10.7  (2025), available at https://www.gap-system.org/Packages/yangbaxter.html.



\end{thebibliography}
\end{document}